\documentclass[11pt]{amsart}
\usepackage[english]{babel}
\usepackage{graphicx}
\usepackage[T1]{fontenc}
\usepackage{color}
\usepackage{bigints}
\usepackage{soul}

\newtheorem{Defi}{Definition}
\newtheorem{Thm}{Theorem}
\newtheorem{Prop}[Thm]{Proposition}
\newtheorem{Cor}{Corollary}[Thm]
\newtheorem{Lem}{Lemma}
\newtheorem{Rem}{Remark}
\newtheorem*{Prop*}{Proposition}
\newtheorem*{Cor*}{Corollary}
\newtheorem*{Thm*}{Theorem}

\def\R{{\mathbb R}}
\def\N{{\mathbb N}}

\DeclareMathAlphabet{\mathitbf}{OML}{cmm}{b}{it}

\newcommand{\dsp}{\displaystyle}

\begin{document}
\title[Existence  and uniqueness analysis of a non-isothermal cross-diffusion system]{Existence and uniqueness analysis of a non-isothermal cross-diffusion system of Maxwell-Stefan type}
\bibliographystyle{plain}

\author[Harsha Hutridurga]{Harsha Hutridurga}
\address{H.H.: Department of Mathematics, Imperial College London, London, SW7 2AZ, United Kingdom.}
\email{h.hutridurga-ramaiah@imperial.ac.uk}

\author[Francesco Salvarani]{Francesco Salvarani}
\address{F.S.: Universit\'e Paris-Dauphine, Ceremade, UMR CNRS 7534, F-75775
Paris Cedex 16, France \& Universit\`a degli Studi di Pavia, Dipartimento di
Matematica, I-27100 Pavia, Italy} 
\email{francesco.salvarani@unipv.it}

\begin{abstract}
In this article we prove local-in-time existence and uniqueness of solution to a non-isothermal cross-diffusion system with Maxwell-Stefan structure.
\end{abstract}

\maketitle

\section{Introduction}\label{sec:introduction}

The Maxwell-Stefan diffusion equations provide an accurate description of diffusive phenomena in gaseous mixtures. Despite its wide use in applied contexts (see \cite{taylor1993multicomponent, 
kri-wes-97}), the mathematical study of the Maxwell-Stefan system is quite a recent subject (we refer to \cite{gio_book, bot-11, jun-ste-13,  bou-gre-sal-12, che-jun-15, mcl-bou-14, bou-gre-sal-15, bou-gre-pav-16, hur-sal-16} and the references therein).

\textcolor{black}{Usually, the derivation of the Maxwell-Stefan diffusion equation starting from the Boltzmann system for mixtures has been obtained by supposing the absence of advective phenomena and with uniform and constant temperature \cite{bou-gre-sal-15, bou-gre-pav-16, hur-sal-16}. However,}
applications often require us to take into account the effects of the fluctuations in temperature on the diffusive behaviour of gaseous mixtures.
In a previous note \cite{HHFS2}, we derived a system which models diffusive phenomena in a non-isothermal context, whose precise form is the following:
\begin{equation}\label{eq:coupled-system-T-known}
\left\{
\begin{array}{ll}
\displaystyle
\partial_t c_i + \nabla_x\cdot J_i = 0 &\qquad \mbox{ in } (0,\infty)\times \Omega, \quad i=1,\dots,n\\ \\
\displaystyle \nabla_x \left( c_i T \right) = -\sum_{j\neq i} k_{ij} \left(c_j J_i-c_i J_j\right)& \qquad
\mbox{ in } (0,\infty)\times \Omega, \quad i=1,\dots,n\\
\displaystyle \sum_{i=1}^n J_i = - \alpha \nabla c_{\rm tot} &\qquad
\mbox{ in } (0,\infty)\times \Omega,
\end{array}
\right.
\end{equation}
for the unknown concentrations $c_i(t,x)$, $i=1,\dots,n$ and for the unknown fluxes $J_i(t,x)$, $i=1,\dots,n$. The spatial domain $\Omega\subset \R^d$ with $d\in \N$ is assumed to be bounded and with regular boundary (at least of class $C^1$). The coefficients $k_{ij}\in\R$ are strictly positive for $i\neq j$, $i,j=1,\dots,n$ and are symmetric, i.e. $k_{ij}=k_{ji}$. \textcolor{black}{Note that the diagonal elements $k_{ii}$ for $i=1,\dots,n$ do not play any role in the flux-gradient relations in \eqref{eq:coupled-system-T-known}.} The coefficient $\alpha\in\R$ is strictly positive. The function $T(t,x)$ represents the unknown local temperature of the system and by $c_{\rm tot}(t,x)$ we denote the total concentration of the mixture, i.e.
$$
c_{\rm{tot}}(t,x):= \sum_{i=1}^n c_i(t,x).
$$
This model will be supplemented with suitable initial conditions $(c^{\rm{in}}_1,\dots, c^{\rm{in}}_n)$, $T^{\rm{in}}(x)$,  and with homogeneous Neumann boundary conditions:
$$
\nabla c_i(t,x)\cdot {\bf n}(x) =0, \qquad (t,x)\in (0,\infty)\times \partial\Omega, \qquad i=1,\dots, n.
$$
In \cite{HHFS2}, we have shown that it is possible to formally decouple the behaviour of the unknown quantities $c_{\rm tot}(t,x)$ and $T(t,x)$ from system (\ref{eq:coupled-system-T-known}): they satisfy the coupled system (see \cite[Lemma 2]{HHFS2})
\begin{align}
& \partial_t c_{\rm tot} - \alpha \Delta c_{\rm tot} = 0 \qquad \qquad \qquad \qquad \qquad \qquad \qquad \quad \mbox{ in } (0,\infty)\times\Omega,\label{eq:evol-total-conc}
\\[0.2 cm]
& \partial_t T - \left( \frac{2}{3} \partial_t \log c_{\rm tot} \right) T - \left( \frac{5\alpha}{3} \nabla \log c_{\rm tot} \right)\cdot \nabla T = 0
\quad\, \mbox{ in } (0,\infty)\times\Omega,\label{eq:evol-temperature}
\end{align}
with initial conditions
$$
c_{\rm tot}(0,x)= c_{\rm tot}^{\rm{in}}(x) :=  \sum_{i=1}^n c_i^{\rm{in}}(x); \qquad  T(0,x)=T^{\rm{in}}(x) \qquad x\in\Omega,
$$
for $i=1,\dots, n$.\\
In the decoupled system \eqref{eq:evol-total-conc}-\eqref{eq:evol-temperature}, the evolution of the total concentration field is governed by Fickian diffusion and the temperature field satisfies an advection equation where the temperature field is being advected by the gradients of the concentration field. As the advective field is of zero normal flux on the boundary $\partial\Omega$, we do not need to impose any type of boundary data for the temperature field. In the regime described by the above decoupled model, the variations of temperature with respect to space and time are completely driven by the variations of the total concentration field.\\
We borrow the following result from \cite{HHFS2} about the coupled system \eqref{eq:evol-total-conc}-\eqref{eq:evol-temperature}:
\begin{Prop}\label{prop:max-princ-cT}
Suppose the initial data $(c_{\rm tot}^{\rm{in}}, T^{\rm{in}})$ to the evolution equations \eqref{eq:evol-total-conc}-\eqref{eq:evol-temperature} are non-negative and satisfy
\begin{align*}
0 < c_{\rm{min}} \le c_{\rm tot}^{\rm{in}}(x) \le c_{\rm{max}} <\infty; \quad
0 < T_{\rm{min}} \le T^{\rm{in}}(x) \le T_{\rm{max}} <\infty.
\end{align*}
Then
\begin{align*}
c_{\rm{min}} \le c_{\rm{tot}}(t,x) \le c_{\rm{max}} \qquad (t,x)\in [0,\infty)\times\Omega.
\end{align*}
Furthermore
\begin{align}\label{eq:temp-solution}
T(t,x) = T^{\rm in}({\scriptstyle X}(0;t,x)) e^{\frac23\int\limits_0^t \partial_t \left(\log c_{\rm tot}\right) (s,{\scriptscriptstyle X}(s;t,x))\, {\rm d}s} \qquad \mbox{ for } (t,x)\in [0,\infty)\times\Omega,
\end{align}
where ${\scriptstyle X}(s;t,x)$ is the flow associated with the vector field $\mathcal{V}(t,x):= -\frac{5\alpha}{3}\nabla \log c_{\rm tot}$.
\end{Prop}
The $n$ mass balance equations in \eqref{eq:coupled-system-T-known} can be compactly written as
\begin{align}\label{eq:compact-mass-balance}
\partial_t {\bf c} + {\rm div}_x\, {\bf J} = 0
\end{align}
where we have used the following notations:
\[
\mbox{Concentration vector: } {\bf c} = \left( c_1, \dots, c_n \right)^\top \in \R^n;
\qquad
\mbox{Flux matrix: }{\bf J} = \left( J_1, \dots, J_n \right)^\top \in \R^{n\times d}.
\]
The operation of divergence on the matrix ${\bf J}$ should be understood as taking divergence of each row vector of ${\bf J}$, thus the end result ${\rm div}_x\, {\bf J}\in \R^n$. Next, consider the flux-gradient relations in \eqref{eq:coupled-system-T-known}. Define matrices ${\bf D}\in \R^{n\times d}$ and ${\bf F}\in\R^{n\times n}$ as
\begin{align*}
{\bf D}_{ij} := \frac{\partial (c_i T)}{\partial x_j}\quad i=1,\dots,n\quad j=1,\dots, d;
\qquad
{\bf F}_{ij} := \left\{
\begin{array}{cc}
k_{ij} c_i & \mbox{ for }j\neq i,\, i,j=1,\dots,n.
\\[0.2 cm]
\dsp -\sum_{\textcolor{black}{r\neq i}} k_{i\textcolor{black}{r}} c_{\textcolor{black}{r}} & \mbox{ for }j = i=1,\dots,n.
\end{array}\right.
\end{align*}
Then, the flux-gradient relations in \eqref{eq:coupled-system-T-known} can be compactly written as
\begin{align}\label{eq:compact-flux-gradient}
{\bf D} = {\bf F}{\bf J}.
\end{align}
Note that $\mbox{Ker}({\bf F}^\top) = \mbox{span}\{{\bf 1}\}$ with ${\bf 1}=\left(1,\dots,1\right)^\top\in\R^n$, because of the symmetry $k_{ij} = k_{ji}$ for all $i,j=1,\dots,n$. \textcolor{black}{The linear dependence of the flux-gradient relations in \eqref{eq:coupled-system-T-known} implies that the columns of ${\bf D}$ belong to $\{{\rm span}\{{\bf 1}\}\}^\perp$.} Fredholm Alternative, thus, implies that we can solve for ${\bf J}$ in terms of ${\bf D}$. With regard to inverting the relation \eqref{eq:compact-flux-gradient}, we shall follow the lead in \cite{bot-11, jun-ste-13} and apply the Perron-Frobenius theory to the quasi-positive matrix ${\bf F}$. Next, we record some spectral properties of the matrix ${\bf F}$ adapted from \cite{jun-ste-13} (we refer to \cite[Lemma 2.1 and Lemma 2.2 on pp.2426-2427]{jun-ste-13} for detailed proof).
\begin{Lem}\label{lem:spectrum-F}
Let $\delta:= c_{\rm{min}} \min_{i,j=1,\dots,n,\, i\neq j} k_{ij} >0$ and let $\eta := 2 c_{\rm max} \sum_{i,j=1, j\neq i}^n k_{ij}$. Then the spectrum of $-{\bf F}$ satisfies
\[
\sigma(-{\bf F}) \subset \{0\} \cup \big[\delta , \eta \big).
\]
Let $\widetilde{\bf F} := {\bf F}\big|_{im({\bf F})}$. Then, $\widetilde{\bf F}$ is invertible on the image $im({\bf F})$. Furthermore, the spectrum of $-\widetilde{\bf F}$ satisfies
\[
\sigma(-\widetilde{\bf F}) \subset \big[\delta , \eta \big).
\]
\end{Lem}
The next result gives a parabolic problem (reduced in dimension) which is equivalent to the Maxwell-Stefan model \eqref{eq:coupled-system-T-known}.
\begin{Thm}\label{Thm:reduction}
Let ${\bf c}, {\bf J}$ be the solution to the non-isothermal Maxwell-Stefan diffusion model \eqref{eq:coupled-system-T-known} \textcolor{black}{and let $c_{\rm tot}$, $T$ be the known solution to the associated decoupled system.} Let ${\bf c}':=(c_1, \dots, c_{n-1})^\top$ and let ${\bf J}':=(J_1,\dots,J_{n-1})^\top$. Then we have
\begin{align}\label{eq:bf-J-prime}
{\bf J}'
= -T\, {\bf F}_0^{-1} \nabla {\bf c}'
- {\bf F}_0^{-1} {\bf c}' \otimes \nabla T 
- \alpha\, {\bf F}_0^{-1} \widetilde{\bf c}' \otimes \nabla c_{\rm tot}
\end{align}
where the matrix ${\bf F}_0\in\R^{(n-1)\times(n-1)}$ has the elements 
\begin{equation}\label{eq:elements-F_0}
\left[{\bf F}_0\right]_{ij} :=
\left\{
\begin{array}{cl}
-\left( k_{ij} - k_{in}\right)c_i & i\neq j,\, i,j=1,\dots,n-1
\\[0.3 cm]
\dsp \sum_{j\ne i} \left( k_{ij} - k_{in}\right)c_j + c_{\rm tot} k_{in} & i=j=1,\dots, n-1
\end{array}\right.
\end{equation}
and  $\widetilde{\bf c}'_{i} = k_{in}c_i$ for $i=1,\dots,(n-1)$. Furthermore, solving the Maxwell-Stefan system \eqref{eq:coupled-system-T-known} is equivalent to solving the following quasi-linear parabolic system for the concentration vector ${\bf c}'(t,x)$:
\begin{align}\label{eq:quasilinear-parabolic}
\partial_t {\bf c}'
-
{\rm div} \left( T\, {\bf B} \nabla {\bf c}' \right)
=
{\bf r} ({\bf c}')
\end{align}
where ${\bf B}:= {\bf F}_0^{-1}$ and the lower order term
\begin{align}\label{eq:lower-order-term}
{\bf r} ({\bf c}')
=
{\rm div} 
\left(  
{\bf B}\left( {\bf c}' \otimes \nabla T \right)
\right)
+
\alpha\, 
{\rm div} 
\left(  
{\bf B}\left( \widetilde{\bf c}' \otimes \nabla c_{\rm tot} \right)
\right)
-
\partial_t c_{\rm tot}.
\end{align}
\end{Thm}

\begin{proof}
From Lemma \ref{lem:spectrum-F}, we have that the matrix ${\bf F}$ in the flux-gradient relation \eqref{eq:compact-flux-gradient} can be inverted on $im({\bf F}) = \{{\rm span}\{{\bf 1}\}\}^\perp$. Note, however, that the column vectors of the flux matrix ${\bf J}$ in \eqref{eq:compact-flux-gradient} do not belong to $\{{\rm span}\{{\bf 1}\}\}^\perp$ (unless $c_{\rm tot} $ is constant, i.e. in the standard Maxwell-Stefan case) because of the closure relation in \eqref{eq:coupled-system-T-known}: in terms of the matrix elements we have
\[
\sum_{i=1}^n {\bf J}_{ij} + \alpha \partial_j c_{\rm tot} = 0
\qquad
\mbox{ for each }j=1,\dots,n.
\]
Next, define a matrix
\begin{align*}
\widetilde{\bf J} := \left( \widetilde{J}_1, \dots, \widetilde{J}_n \right)^\top
\quad
\mbox{ with the vectors }
\quad
\widetilde{J}_i := 
\left\{
\begin{array}{cl}
J_i & \mbox{for }i\neq n
\\[0.2 cm]
J_n + \alpha \nabla c_{\rm tot} & \mbox{otherwise. }
\end{array}\right.
\end{align*}
Observe that, by construction, the columns of the new matrix $\widetilde{\bf J}$ belong to $\{{\rm span}\{{\bf 1}\}\}^\perp$. We are essentially going to rewrite the flux-gradient relation \eqref{eq:compact-flux-gradient} as
\begin{align}\label{eq:tilde-compact-flux-gradient}
\widetilde{\bf D} := {\bf D} + {\bf F} {\bf A} = {\bf F} \widetilde{\bf J}
\end{align}
where the matrix ${\bf A} := \widetilde{\bf J} - {\bf J}$. The characterization of $im({\bf F})$ and the definition of $\widetilde{\bf D}$ also suggest that the columns of the matrix $\widetilde{\bf D}$ belong to $\{{\rm span}\{{\bf 1}\}\}^\perp$. Following the work of J{\"u}ngel and Stelzer \cite[p.~2428]{jun-ste-13}, we eliminate the $n^{\rm th}$ component in the concentration vector. We reduce the system of $n$ components, into a system of $n-1$ components. To that end, define the matrix ${\bf X}$ and its inverse with the elements
\begin{equation*}
{\bf X}_{ij} :=
\left\{
\begin{array}{cl}
1 & i=j=1,\dots,n
\\[0.2 cm]
-1 & i=n,\, j=1,\dots, n-1
\\[0.2 cm]
0 & \mbox{otherwise.}
\end{array}\right.
\qquad \quad
{\bf X}^{-1}_{ij} :=
\left\{
\begin{array}{cl}
1 & i=j=1,\dots,n
\\[0.2 cm]
1 & i=n,\, j=1,\dots, n-1
\\[0.2 cm]
0 & \mbox{otherwise.}
\end{array}\right.
\end{equation*}
Acting ${\bf X}^{-1}$ on the left of $\widetilde{\bf D}$ and $\widetilde{\bf J}$ yields
\[ {\bf X}^{-1} \widetilde{\bf D} = 
\left( 
\begin{array}{c}
\nabla({\bf c}'T) + \alpha\, \widetilde{\bf c}' \otimes \nabla c_{\rm tot}
\\[0.2 cm]
0 
\end{array} 
\right);
\qquad
{\bf X}^{-1} \widetilde{\bf J} = 
\left( 
\begin{array}{c}
{\bf J}'
\\
0 
\end{array} 
\right)
\] 
where $\widetilde{\bf c}'\in\R^{n-1}$ with elements $\widetilde{\bf c}'_{i} = k_{in}c_i$ for $i=1,\dots,(n-1)$. A computation also yields
\[ {\bf X}^{-1} {\bf F} {\bf X} = 
\left( 
\begin{array}{cc}
- {\bf F}_0 & \widetilde{\bf c}'
\\
0 & 0 
\end{array} 
\right)
\] 
with the matrix ${\bf F}_0$ given by \eqref{eq:elements-F_0}. Using the similarity between the blockwise upper-triangular matrix $-{\bf X}^{-1}{\bf F}{\bf X}$ and $-{\bf F}$, we can deduce that ${\bf F}_0$ is invertible using Lemma \ref{lem:spectrum-F} (see Lemma \ref{lem:spectrum-B} for further details). Next, we act ${\bf X}^{-1}$ on the left of the modified flux-gradient relation \eqref{eq:tilde-compact-flux-gradient} yielding
\begin{align*}
{\bf X}^{-1} \widetilde{\bf D}
=
\left( 
\begin{array}{c}
\nabla({\bf c}'T) + \alpha\, \widetilde{\bf c}' \otimes \nabla c_{\rm tot}
\\[0.2 cm]
0 
\end{array} 
\right)
=
{\bf X}^{-1} {\bf F} {\bf X}
{\bf X}^{-1} \widetilde{\bf J} 
=
\left( 
\begin{array}{cc}
- {\bf F}_0 & \widetilde{\bf c}'
\\
0 & 0 
\end{array} 
\right)
\left( 
\begin{array}{c}
{\bf J}'
\\
0 
\end{array} 
\right)
\end{align*}
from which we deduce the relation \eqref{eq:bf-J-prime}. Now, act ${\bf X}^{-1}$ on the left of the continuity equation \eqref{eq:compact-mass-balance} yielding
\[
\partial_t {\bf c}' + {\rm div} {\bf J}' = 0.
\]
Substituting \eqref{eq:bf-J-prime} for ${\bf J}'$ in the above equation yields the quasi-linear parabolic equation \eqref{eq:quasilinear-parabolic}.
\end{proof}

\section{Local existence and uniqueness result}
Our strategy to prove well-posedness of the model problem \eqref{eq:coupled-system-T-known} is to prove that a unique solution exists for the reduced quasi-linear parabolic system \eqref{eq:quasilinear-parabolic} (see Theorem \ref{Thm:reduction} which shows that these two systems are equivalent). To this end, we need the notion of an operator being \emph{normally elliptic}. Next, we give the definition of this notion. For further details, we suggest the book chapter \cite{amann1993nonhomogeneous}. We shall then apply some well-known results for quasi-linear parabolic systems based on $\mathrm L^p$-maximal regularity.
\begin{Defi}[Normally Elliptic]
A linear second order operator $\mathcal{A}u:= - \partial_j (a_{jk}(t,x) \partial_k u + b_j u)$ on $\Omega\subset\R^d$ is said to be \emph{normally elliptic} if the associated principal symbol $a_\pi(x,\xi):= a_{jk}\xi^j\xi^k$ for $\xi\in\mathbb{S}^{d-1}$ has a spectrum away from zero, i.e.
\begin{equation}\label{eq:normally-elliptic-spectrum}
\sigma\left( a_{\pi}(x,\xi) \right) \subset \left\{ z\in\mathbb{C} \, \, {\rm s.t. }\,\, {\rm Re }z>0 \right\}.
\end{equation}
\end{Defi}
\begin{Rem}
In the quasi-linear setting, i.e. when the coefficients $a_{jk}(t,x,u)$ in the differential operator depends on the solution $u(t,x)$, the notion of \emph{normal ellipticity} should be interpreted as follows: the associated principal symbol has a spectrum away from zero -- i.e. to satisfy \eqref{eq:normally-elliptic-spectrum} -- for each $u\in\mathbb{E}$ where $\mathbb{E}$ is the space in which we look for solutions to the given quasi-linear problem.
\end{Rem}
Our next task is to prove local (in time) existence-uniqueness result for the reduced system \eqref{eq:quasilinear-parabolic}. For readers' convenience, we shall recall the notion of strong solutions to the quasi-linear problem \eqref{eq:quasilinear-parabolic} in the $\mathrm L^p$-sense.
\begin{Defi}\label{def:solution-notion}
A function ${\bf v}(t,x)$ defined on $[0,\ell)\times\Omega$ is said to be a strong solution to \eqref{eq:quasilinear-parabolic} if
\begin{align*}
& {\bf v}\in C([0,\ell);[\mathrm L^1(\Omega)]^{n-1}) \cap [\mathrm L^\infty ([0, \ell-\tau]\times\Omega)]^{n-1},
\quad
\forall \tau\in (0,\ell),\\
& \forall p\in [1,\infty),
\, \partial_t {\bf v}, \, \partial_{x_k}{\bf v},\, \partial^2_{x_kx_l}{\bf v}\in [\mathrm L^p((\tau, \ell-\tau)\times\Omega)]^{n-1}\quad \forall k,l\in\{1,\dots,d\}
\end{align*}
and ${\bf v}(t,x)$ solves \eqref{eq:quasilinear-parabolic} a.e. in $(0,\ell)\times\Omega$.
\end{Defi}
\textcolor{black}{Thanks to the strict positivity of $\delta$ and $\eta$ (see Lemma \ref{lem:spectrum-F} for their definitions), an immediate consequence of \cite[Lemma 2.3 on p.~2428]{jun-ste-13}, concerning the spectrum of the matrix ${\bf B}:={\bf F}_0^{-1}$ is}
\begin{Lem}\label{lem:spectrum-B}
The matrix ${\bf F}_0$ given in \eqref{eq:elements-F_0} is invertible with spectra
\[
\sigma\left( {\bf F}_0 \right) \subset [\delta, \eta);
\qquad
\sigma\left( {\bf F}_0^{-1} \right) \subset (\eta^{-1}, \delta^{-1}].
\]
\end{Lem}
Note that Lemma \ref{lem:spectrum-B} implies that the matrix $T{\bf B}$ has a spectrum away from zero as the temperature field $T(t,x)$ is bounded away from zero -- see Proposition \ref{prop:max-princ-cT}. Next, note that the spectrum associated with the principal symbol of the quasi-linear operator in the reduced system \eqref{eq:quasilinear-parabolic} is nothing but the spectrum of the matrix $T{\bf B}$. This implies that the quasi-linear operator in \eqref{eq:quasilinear-parabolic} is normally elliptic. As the coefficients in the reduced system \eqref{eq:quasilinear-parabolic} are all bounded, the local-in-time existence of solution follows. We record the main result of this note below.
\begin{Thm}
Let the domain $\Omega\subset\R^d$ be a bounded domain with smooth boundary. Let the initial data $(c^{\rm in}_1, \dots, c^{\rm in}_{n-1})$ to the quasi-linear problem \eqref{eq:quasilinear-parabolic} be non-negative measurable functions such that
\[
\sum_{i=1}^{n-1} c^{\rm in}_i(x) \le c_{\rm max}
\]
where $c_{\rm max}$ is the upper bound for the total concentration. Then, there exists a unique local-in-time solution -- in the $\mathrm L^p$-sense (see Definition \ref{def:solution-notion}) -- to the system \eqref{eq:quasilinear-parabolic}.
\end{Thm}
\textcolor{black}{Theorem \ref{Thm:reduction} proves that the non-isothermal Maxwell-Stefan system \eqref{eq:coupled-system-T-known} and the reduced quasi-linear system \eqref{eq:quasilinear-parabolic} are equivalent. Hence we have the following
\begin{Cor}
Let $c^{\rm in}_1, \dots, c^{\rm in}_{n}$ be $n$ non-negative functions belonging to $\mathrm L^\infty(\Omega)$ such that  
\[
\sum_{i=1}^n c^{\rm in}_i(x) \le c_{\rm max}.
\]
Let $T^{\rm in}\in \mathrm L^\infty(\Omega)$ be a non-negative initial temperature field. Then, with $(c^{\rm in}_1, \dots, c^{\rm in}_{n})$ and $T^{\rm in}$ as initial datum, there exists a unique local-in-time solution to the non-isothermal Maxwell-Stefan system \eqref{eq:coupled-system-T-known}. 
\end{Cor}
}
\noindent {\bf Acknowledgments:} The authors are grateful to the referees for their useful comments and suggestions.
This work was partially funded by the projects \textit{Kimega} (ANR-14-ACHN-0030-01) and
\textit{Kibord} (ANR-13-BS01-0004). The first author acknowledges the support of the EPSRC programme grant ``Mathematical fundamentals of Metamaterials for multiscale Physics and Mechanic'' (EP/L024926/1) and the ERC grant MATKIT (ERC-2011-StG).

\bibliography{biblio}

\end{document}